\documentclass[a4paper, 11pt, reqno]{amsart}

\usepackage[top=1in, bottom=1.25in, left=1in, right=1in]{geometry}

\allowdisplaybreaks

\numberwithin{equation}{section}

\usepackage[latin1]{inputenc}
\usepackage[T1]{fontenc}
\usepackage[english]{babel}
\usepackage{latexsym,amsmath,amssymb,amsfonts}
\usepackage{dsfont}
\usepackage{color}
\usepackage{graphicx}
\usepackage{float}
\usepackage{esint}
\usepackage{enumerate}

\theoremstyle{plain}
\begingroup
\newtheorem{theorem}{Theorem}[section]
\newtheorem{lemma}[theorem]{Lemma}

\endgroup
\theoremstyle{definition}
\begingroup
\newtheorem{definition}[theorem]{Definition}
\newtheorem{remark}[theorem]{Remark}

\endgroup
\theoremstyle{remark}

\newcommand{\N}{\mathbb N}
\newcommand{\Z}{\mathbb Z}

\newcommand{\R}{\mathbb R}

\DeclareMathOperator*{\esssup}{ess\,sup}

\newcommand{\dd}{\mathrm{d}}

\newcommand{\dx}{\;\dd x}
\newcommand{\dy}{\;\dd y}

\newcommand{\e}{\varepsilon}
\newcommand{\del}{\delta}

\renewcommand{\o}{\Omega}

\newcommand{\per}{\mathcal{P}}

\newcommand{\average}{{\mathchoice {\kern1ex\vcenter{\hrule height.4pt
width 6pt
depth0pt} \kern-9.7pt} {\kern1ex\vcenter{\hrule height.4pt width 4.3pt
depth0pt}
\kern-7pt} {} {} }}

\title[A Note on Homogenization Effects on Phase Transition Problems]{A Note on Homogenization Effects on Phase Transition Problems}
\author[Adrian Hagerty]{Adrian Hagerty}

\begin{document}

\begin{abstract}
In fluid-fluid phase transitions problems featuring small scale heterogeneity, we see that when the scale heterogeneity is sufficiently small, the periodic potential function $W(x,p)$ can be replaced with a homogenized potential function $W_H(p)$. This allows for a reduction to a problem already studied by Fonseca and Tartar \cite{fonseca89}. In particular, we see an isotropic transition potential, in contrast with the case where the scales are roughly comensurate \cite{CFHP18}. This note is intended as a prelude to a more complete analysis of the full range of scales.
\end{abstract}

\maketitle

\section{Introduction}

In this manuscript, we remark upon a particular case in the study of  interaction between fluid-fluid phase transition and homogenization in the presence of small scale heterogeneities in the fluid. This is part of an ongoing project to understand the limiting behavior of the variational problem
\[
\mathcal{F}_{\e,\delta}(u) := \int_\o \left[\, \frac{1}{\del} W \left( \frac{x}{\e}, u(x) \right)
    + \del |\nabla u(x)|^2 \,\right]\dx,
\] 
where $W: \R^N \times \R^d \to [0, \infty)$ is a double-well potential that is 1-periodic in its first argument. Here, the periodicity at scale $\e$ fixes the scaling of the heterogeneity while $\del$ corresponds to the thickness of our transition layers.

We characterize the limiting behavior of minimizers to $\mathcal{F}_{\e,\del}$ by identifying the $\Gamma$-limit of $\mathcal{F}_{\e,\delta}$ as $\e,\delta \to 0$ for different regimes corresponding to the relative behavior of $\e$ and $\del$.

In \cite{CFHP18}, this problem was studied in the case where $\e$ and $\del$ are commensurate. This was incorporated by assuming that $\del = \e$, that is to say
\[
\mathcal{F}_{\e}(u) := \int_\o \left[\, \frac{1}{\e} W \left( \frac{x}{\e}, u(x) \right)
    + \e |\nabla u(x)|^2 \,\right]\dx.
\]
In this regime, the $\Gamma$-limit was an anisotropic perimeter caused by potential mismatch between the direction of periodicity and the orientation of the interface. 

The limiting behavior when the phase transitions occur at a finer scale than the homogenization, $\delta << \e$, will be the subject of a forthcoming publication, currently under preparation.

In this paper we study a scaling in which the homogenization effects occur far more rapidly than that of the phase transition, namely $\e << \delta$. For our key lemma, we will need to require a certain quantitative control to this scale, namely 
\[ \frac{\e}{\del^{\frac{3}{2}}} \to 0.\]
An identical scaling is observed in the paper of Ansini, Braides and Piat \cite{AnsBraChi2} who consider energies of the form
\[ \int_{\o} \left[ \frac{1}{\del} W(u(x)) + \del f \left( \frac{x}{\e} , \nabla u \right) \right] \dx.\]
In their consideration of scalings of $\e$ finer than $\del$ in Section 4.3, they require the explicit relationship $\del << \e \sqrt \e$. It is not yet clear if this scaling is a necessary feature of problems incorporating fine scale homogenization or merely a technical consideration.

In the presence of this rapid periodicity, we can pass from the periodic potential function $W(x, p)$ to a homogenized potential function $W_H(p)$ depending only on the value of the function $u$ and not on position. This will allow us to compare the limiting behavior to the well-studied case of functionals of the form
\[
\mathcal{F}_{\delta}(u) := \int_\o \left[\, \frac{1}{\del} W \left( u(x) \right)
    + \del |\nabla u(x)|^2 \,\right]\dx,
\] 
as found in the work of Fonseca and Tartar \cite{fonseca89}.


\subsection{Statement of the main results}\label{sec:mainresult}

In the following, $Q\subset\R^N$ denotes the unit cube centered at the origin with faces orthogonal to the coordinate axes, $Q:=(-1/2,1/2)^N$. The set $\o \subset \R^N$ will always be a bounded, open domain with Lipschitz boundary. 

Consider a double well potential $W:\R^N\times\R^d\to[0,\infty)$ satisfying the following properties:
\begin{itemize}
\item [(H0)] $x\mapsto W(x,p)$ is $Q$-periodic for all $p\in\R^d$,
\item [(H1)] $W$ is a Carath\'{e}odory function, \emph{i.e.},
\begin{itemize}
\item[(i)] for all $p\in\R^d$ the function $x\mapsto W(x,p)$ is measurable, 
\item[(ii)] for a.e. $x\in Q$ the function $p\mapsto W(x,p)$ is continuous,
\end{itemize}
\item [(H2)] there exist $a,b\in\R^d$ such that $W(x,p)=0$ if and only if $p\in\{a,b\}$, for a.e. $x\in Q$,
\item[(H3)] there exists a continuous function $W_c:\R^d\to[0,\infty)$ such that $W_c(p)\leq W(x,p)$ for a.e. $x\in Q$ and $W_c(p)=0$ if and only if $p\in\{a,b\}$.
\item [(H4)] there exist $C>0$ and $q\geq2$ such that $\frac{1}{C}|p|^q-C\leq W(x,p)\leq C(1+|p|^q)$ for a.e. $x\in Q$ and all $p\in\R^d$.
\item [(H5)] $W$ is locally Lipschitz in $p$, that is, for every $K \subset \R^d$ compact there is a constant $L$ such that
\[ |W(x,p) - W(x,q)| \leq L|p-q| \]
for almost every $x \in Q$ and every $p,q \in K$.
\end{itemize}

In the case where $\e << \del$, the homogenization effects occur so rapidly that the system is essentially homogenized before interacting with the phase transition problem. In this case, we prove that the $\Gamma$-limit of $\mathcal{F}_{\e,\del}$ coincides with the interfacial energy associated with a homogenized potential.

\begin{definition}
We define the functional $F^H_0:L^1(\o;\R^d)\to[0,+\infty]$ as
\begin{equation}\label{eq:fHom0}
F^H_0(u):=
\begin{cases}
K_H \per(\{u = a\}; \o) & \text{ if } u \in BV(\o; \{a,b\}),\\
&\\
+\infty & \text{ otherwise}.
\end{cases}
\end{equation}
Here the transition energy density $K_H$ is defined as
\begin{equation}\label{eq:Kh}
K_H := 2 \inf \left\{ \int_0^1 \sqrt{W_H(g(s))} |g'(s)| \dd s : g\in C^1_{pw}([0,1];\R^d;a,b)
    \,\right\},
\end{equation}
where $C^1_{pw}([0,1];\R^d;a,b)$ denotes the space of piecewise $C^1$ curves from $[0,1]$ to $\R^d$ such that $g(0)=a$ and $g(1)=b$, and the homogenized potential $W_H:\R^d\to[0,+\infty)$ is given by
\begin{equation}\label{eq:homPot}
W_H(p) := \int_Q W(y, p) \dy 
\end{equation}
\end{definition}

The main result of this paper is the following $\Gamma$-convergence result in the case where the homogenization parameter $\e$ is sufficiently small with respect to the phase transition parameter $\delta$.

\begin{theorem}\label{thm:homFirst}
Let $\{\e_n\}_{n\in\N}$ $\{\del_n\}_{n\in\N}$ be two infinitesimal sequences such that
\[
\lim_{n\to\infty}\frac{\del_n^{\frac{3}{2}}}{\e_n} \to +\infty.
\]
Set $F_{n}:=\mathcal{F}_{\e_n,\del_n}$.
Assume that $W$ satisfies hypotheses (H0)-(H4).
Then the following hold:
\begin{enumerate}
\item If $\{u_n\}_{n\in\N} \subset H^1(\o; \R^d)$ is such that
\[
\sup_{n \in \mathbb{N}} F_n(u_n) < +\infty,
\]
then, up to a subsequence (not relabeled), we have $u_n \to u$ in $L^1(\o; \R^d)$ for some $u \in BV(\o; \{a,b\})$.
\item As $n \to \infty$, we have $F_n \stackrel{\Gamma-L^1}{\longrightarrow} F^H_0$.
\end{enumerate}

\end{theorem}


\section{Preliminaries}\label{sec:prel}

In this section we collect basic notions needed in the paper.

\subsection{Sets of finite perimeter}

We recall the definition and some well known facts about sets of finite perimeter (we refer the reader to \cite{AFP} for more details).

\begin{definition}
Let $E\subset\R^N$ with $|E|<\infty$ and let $\o\subset\R^N$ be an open set.
We say that $E$ has \emph{finite perimeter} in $\o$ if
\[
P(E;\o):=\sup\left\{\, \int_E \mathrm{div}\varphi \,\dd x \,:\, \varphi\in C^1_c(\o;\R^N)\,,\, \|\varphi\|_{L^\infty}\leq1  \,\right\}<\infty\,.
\]
\vspace{0\baselineskip}
\end{definition}

\begin{remark}\label{rem:defvar}
$E\subset\R^N$ is a set of finite perimeter in $\o$ if and only if $\chi_E\in BV(\o)$, \emph{i.e.}, the distributional derivative $D\chi_E$ is a finite vector valued Radon measure in $\o$, with
\[
\int_{\R^N} \varphi \,\dd D\chi_E=\int_E \mathrm{div}\varphi \,\dd x
\]
for all $\varphi\in C^1_c(\o;\R^N)$, and $|D\chi_E|(\o)=P(E;\o)$. \vspace{0.8\baselineskip}
\end{remark}

\begin{remark}
Let $\o\subset\R^N$ be an open set, let $a,b\in\R^d$, and let $u\in L^1(\o;\{a,b\})$.
Then $u$ is a function of \emph{bounded variation} in $\o$, and we write $u\in BV(\o;\{a,b\})$, if the set $\{u=a\}:=\{ x\in \o \,:\, u(x)=a\}$ has finite perimeter in $\o$. \vspace{0.5\baselineskip}
\end{remark}

\begin{definition}\label{def:mun}
Let $E\subset\R^N$ be a set of finite perimeter in the open set $\o\subset\R^N$. We define $\partial^* E$, the \emph{reduced boundary} of $E$, as the set of points $x\in\R^N$ for which the limit
\[
\nu_E(x):=-\lim_{r\to0}\frac{D\chi_E(x+rQ)}{|D\chi_E|(x+rQ)}
\]
exists and is such that $|\nu_E(x)|=1$.
The vector $\nu_E(x)$ is called the \emph{measure theoretic exterior normal} to $E$ at $x$. \vspace{0.5\baselineskip}
\end{definition}


\subsection{$\Gamma$-convergence}

We refer to \cite{Braides} and \cite{dalmaso93} for a complete study of $\Gamma$-convergence in metric spaces.

\begin{definition}\label{def:gc}
Let $(X,\mathrm{m})$ be a metric space. We say that $F_n:X\to[-\infty,+\infty]$ $\Gamma$-converges to $F:X\to[-\infty,+\infty]$,
and we write $F_n\stackrel{\Gamma-\mathrm{m}}{\longrightarrow} F$, if the following hold:
\begin{itemize}
\item[(i)] for every $x\in X$ and every $x_n\to x$ we have
\[
F(x)\leq\liminf_{n\to\infty} F_n(x_n)\,,
\]
\item[(ii)] for every $x\in X$ there exists $\{x_n\}_{n=1}^\infty\subset A$ (so called \emph{a recovery sequence}) with $x_n\to x$ such that
\[
\limsup_{n\to\infty} F_n(x_n)\leq F(x)\,.
\]
\end{itemize} \vspace{0\baselineskip}
\end{definition}


\section{Main Result}\label{sec:homfirst}

We proceed to prove our main result, the $\Gamma$ convergence result in the case where the homogenization occurs at a much smaller scale than the phase transition. To be precise, we consider the scaling
\[
\frac{\e}{\delta^{\frac{3}{2}}} \to 0
\]

\begin{remark}
The reason that this scaling is necessary as opposed to the more general case without a factor of $\frac{3}{2}$ is not yet clear. Indeed, if one could show that a sequence $u_{\e}$ with bounded energy satisfied
\[ \lim_{\e \to 0} F_{\e}(u_{\e}; Q \setminus \{|x_N| > \delta\}) = 0 \]
then this theorem would follow in the more general scaling $\e << \delta$.
\end{remark}

First, in order to rule out possible pathological behavior corresponding to large values of $u$, we will introduce a truncated potential $\widetilde{W}$.

\begin{definition}\label{def:Wtilde}
Let $R>0$ be given such that every minimizing curve $g\in C^1_{pw}([0,1];\R^d;a,b)$ for the minimization problem defining $K_H$ (see \eqref{eq:Kh}) is such that $|g(t)|\leq R$ for every $t\in[-1,1]$. Let
\[
M:=\esssup_{x\in\Omega}\max_{|p|\leq R} W(x,p).
\]
and define the \emph{truncated potential} $\widetilde{W}:\Omega\times\R^d\to[0,\infty)$ as
\[
\widetilde{W}(x,p):=\min\{ W(x,p), M \}.
\]
\end{definition}

\begin{remark}
The truncated potential $\widetilde{W}$ is Lipschitz (not only locally).
Moreover, note that $0<M<+\infty$. Indeed, thanks to the upper bound given by (H4).
\end{remark}

The proof of Theorem \ref{thm:homFirst} is based on a convergence result result (Lemma \ref{lem:wHom})
stating that in the functional $F_n$ it is possible to \emph{substitute} the (truncated) energy with the a homogenized energy. Thus, we introduce the intermediate energy we will be using.

\begin{definition}
We define the \emph{homogenized diffuse energy} $F^{H}_{\e}:L^1(\o;\R^d)\to[0,+\infty]$ by
\[
F^{H}_n(u) = \int_{\o} \left[\, \frac{1}{\del_n} \widetilde{W}_H(u(x)) + \del_n |\nabla u(x)|^2 \,\right] \dx
\]
for $u\in H^1(\o;\R^d)$ and $+\infty$ otherwise. Here $\widetilde{W}_H$ is defined as
\[ \widetilde{W}_H(p) = \int_Q \widetilde{W}(y,p) \dy\]
just as in \ref{eq:homPot}.
\end{definition}

We now prove that as $n \to \infty$, the limiting behavior of $F^H_n$ totally captures the limiting behavior the truncated problem.

\begin{lemma}\label{lem:wHom}
Let $\{u_n\}_{n\in\N} \subset H^1(\o; \R^d)$ be such that
\begin{equation}\label{eq:condition}
\sup_{n \in \N} \int_{\o} \del_n |\nabla u_n|^2 \dx < \infty.
\end{equation}
Then
\[
\lim_{n \to \infty} \left|\, \frac{1}{\del_n}\int_\o
    \left[\, \widetilde{W}\left(\frac{x}{\e_n},u_n(x)\right) - \widetilde{W}_H(u_n(x)) \,\right] \dx \,\right| = 0,
\]
where, for $p\in\R^d$, we set
\begin{equation}\label{eq:widetildeW}
\widetilde{W}_H(p):=\int_Q \widetilde{W}(x,p) \dx.
\end{equation}
\end{lemma}

\begin{proof}
Let
\begin{equation}\label{eq:T}
T:=\sup_{n \in \N} \int_{\o} \del_n |\nabla u_n|^2 \dx < \infty.
\end{equation}
%
Write
\[
\o = \bigcup_{i=1}^{M_n} Q(p_i, \e_n) \cup R_n,
\]
where $p_i \in \e_n \Z^N$, $R_n$ is the set of cubes $Q(z,\e_n)$ with $z\in\e_n\Z^N$ such that $Q(z,\e_n)\cap\partial\o\neq\emptyset$, and $M_n\in\N$.
Note that
\begin{align}
& \left| \frac{1}{\delta_n} \int_{\bigcup_{i=1}^{M_n} Q(p_i, \e_n) }
    \left(  \widetilde{W} \left(\frac{x}{\e_n}, u_n\right) - \widetilde{W}_H(u_n)\right)  \dx  \right| \notag \\
&\hspace{3cm} \leq \frac{1}{\delta_n} \sum_{i=1}^{M_n} \left| \int_{Q(p_i, \e_n) } 
    \left(\, \widetilde{W} \left(\frac{x}{\e_n}, u_n\right) - \widetilde{W}_H(u_n) \,\right) \dx \right| \notag \\
& \hspace{3cm} = \frac{\e_n^N}{\delta_n} \sum_{i=1}^M
    \left| \int_Q  \left(\, \widetilde{W}(y, u_n (p_i + \e_n y )) - \widetilde{W}_H(u_n(p_i + \e_n y) \,\right) \dy \right|.
\label{eq:cubesHom}
\end{align}
where in the last step we have used the substitution $x = p_i + \e y$ and noting that
$\widetilde{W}\left(y - \frac{p_i}{\e_n}, \cdot \right) = \widetilde{W}(y, \cdot)$ by periodicity.
From here, we can rewrite \eqref{eq:cubesHom} as
\begin{align}\label{eq:cubesHom1}
&\frac{\e_n^N}{\delta_n} \sum_{i=1}^{M_n} \left| \int_Q  \int_Q  \left( \widetilde{W}(y, u_n (p_i + \e_n y )) - \widetilde{W}(z, u_n(p_i + \e_n y)) \right) \dd z \dy \right| \notag\\
& \ \ \ \ \ \ = \frac{\e_n^N}{\delta_n} \sum_{i=1}^{M_n} \left| \int_Q  \int_Q  \left( \widetilde{W}(y, u_n (p_i + \e_n y )) - \widetilde{W}(y, u_n(p_i + \e_n z)) \right) \dd z \dy \right| \notag\\
& \ \ \ \ \ \ \leq \frac{\e_n^N}{\delta_n} \sum_{i=1}^{M_n} \int_Q  \int_Q \left|   \widetilde{W}(y, u_n (p_i + \e_n y )) - \widetilde{W}(y, u_n(p_i + \e_n z)) \right| \dd z \dy \notag\\ 
& \ \ \ \ \ \ \leq \frac{L \e_n^N}{\delta_n} \sum_{i=1}^{M_n} \int_Q  \int_Q  \left| u_n (p_i + \e_n y )) - u_n(p_i + \e_n z)) \right| \dd z \dy \notag \\
& \ \ \ \ \ \ \leq \frac{L \e_n^N}{\delta_n} \sum_{i=1}^{M_n} \left( \int_Q  \int_Q  \left| u_n (p_i + \e_n y ) - \overline{u}_{i,n}  \right| \dd z \dy
    + \int_Q  \int_Q  \left|\overline{u}_{i,n} - u_n (p_i + \e_n z )  \right| \dd z \dy \right)
\end{align}
where in the second to last step  $L>0$ is the Lipschitz constant of $\widetilde{W}$, and we define
\[
\overline{u}_{i,n} := \int_{Q} u_n (p_i + \e_n z ) \dd z.
\]
By symmetry, the last term in \eqref{eq:cubesHom1} can be written as
\begin{align*}
\frac{2L \e_n^N}{\delta_n} \sum_{i=1}^{M_n}  \int_Q  \int_Q 
    \left| u_n (p_i + \e_n y ) - \overline{u}_{i,n}  \right| \dd z \dy
    &= \frac{2L}{\delta_n} \sum_{i=1}^{M_n}  \int_Q  \left| u_n (p_i + \e_n y )
        - \overline{u}_{i,n}  \right| \dy.
\end{align*}
By the Poincar\'e inequality
\begin{align}\label{eq:cubesHom2}
\frac{2L}{\delta_n} \sum_{i=1}^{M_n}  \int_Q  \left| u_n (p_i + \e_n y )
        - \overline{u}_{i,n}  \right| \dy
& \leq \frac{2 C L \e_n^{N+1}}{\delta_n} \sum_{i=1}^{M_n}  \int_Q |\nabla u_n(p_i + \e_n y)| \dy
    \notag\\
&=\frac{2CL \e_n}{\delta_n} \sum_{i=1}^{M_n}  \int_{Q(p_i, \e_n)} |\nabla u_n(x)| \dx \notag\\
    &\leq \frac{2CL \e_n}{\delta_n} \int_\o |\nabla u_n| \dx  \notag\\
&\leq \frac{2\tilde{C}L \e_n}{\delta_n} |\o|^{\frac{1}{2}} \left( \int_{\o} |\nabla u_{\e}|^2 \dx \right)^{\frac{1}{2}} \notag\\
& = \frac{2\tilde{C}L \e_n}{\delta_n^{\frac{3}{2}}}
    \left( \delta_n \int_{\o} |\nabla u_n|^2 \dx \right)^{\frac{1}{2}} \notag\\
&\leq \frac{2\tilde{C}L \e_n}{\delta_n^{\frac{3}{2}}} T,
\end{align}
where $T\in[0,\infty)$ is defined in \eqref{eq:T}.
Using \eqref{eq:cubesHom}, \eqref{eq:cubesHom1} and \eqref{eq:cubesHom2}, we conclude that
\begin{equation}\label{eq:cH}
\lim_{n \to \infty}  \left| \frac{1}{\delta_n} \int_{\bigcup_{i=1}^{M_n} Q(p_i, \e_n)}
    \left(  \widetilde{W} \left(\frac{x}{\e_n}, u_n\right) - \widetilde{W}_H(u_n)\right)  \dx  \right|  = 0.
\end{equation}

Noticing that $\widetilde{W}$ and $\widetilde{W_H}$ are bounded and $|R_n|\leq C\e_n$ we get
\begin{equation}\label{eq:cH1}
\lim_{n\to\infty}  \left| \frac{1}{\delta_n} \int_{R_n} \left(  \widetilde{W} \left(\frac{x}{\e_n}, u_n\right) - \widetilde{W}_H(u_n)\right)  \dx  \right|  = 0.
\end{equation}
Thus, from \eqref{eq:cH} and \eqref{eq:cH1} we conclude.
\end{proof}

With Lemma \ref{lem:wHom}, we may proceed to prove the $\Gamma$-convergence result stated in Theorem $\ref{thm:homFirst}$.

\begin{proof}[Proof of Theorem \ref{thm:homFirst}]
\emph{Step 1: compactness.} Let $\{u_n\}_{n\in\N} \subset H^1(\o; \R^d)$ be a sequence with
\[
\sup_{n\in\N}F_n(u_n)<+\infty.
\]
Then we have
\[
\sup_{n \in \N} \int_{\o} \del_n |\nabla u_n|^2 \dx < \infty
\]
and thus, since $\widetilde{W}\leq W$, we can apply Lemma \ref{lem:wHom} to conclude
\[
\sup_{n\in\N}F^H_n(u_n)<+\infty.
\]
Thus, by classical results (see, for instance, \cite[Theorem 4.1]{fonseca89}) we get that, up to a subsequence
$u_n\to u$ in $L^1(\o;\R^d)$ with $u\in BV(\o;\{a,b\})$.\\

\emph{Step 2: liminf inequality.} Let $\{u_n\}_{n\in\N} \subset H^1(\o; \R^d)$ with $u_n\to u$ in $L^1(\o;\R^d)$ . In order to prove 
\[
F^H_0(u) \leq \liminf_{n \to \infty} F_n(u_n).
\]
Since this is vacuously true if the right hand side is positive infinity, without loss of generality, we restrict ourselves to the case where
\begin{equation}\label{eq:liminf}
\lim_{n \to \infty} F_n(u_n)=\liminf_{n \to \infty} F_n(u_n) < + \infty.
\end{equation}
Using Step 1, we get that $u\in BV(\o;\{a,b\})$. Moreover, noticing that, by definition of $M$ (see Definition \ref{def:Wtilde}), we get
\begin{equation}\label{{eq:KHmodified}}
K_H= 2 \inf \left\{ \int_0^1 \sqrt{\widetilde{W}_H(g(s))} |g'(s)| ds : g\in C^1_{pw}([0,1];\R^d;a,b) \,\right\},
\end{equation}
where $\widetilde{W}$ is defined in \eqref{eq:widetildeW}.
Thus, using standard results
(see, for instance, \cite[Theorem 3.4]{fonseca89}), we get
\[
F^H_0(u)\leq \liminf_{n\to\infty} F^H_n(u_n) \leq \liminf_{n\in\N} F_n(u_n),
\]
where in the last step we used Lemma \ref{lem:wHom} noting that \eqref{eq:liminf} yields the validity of \eqref{eq:condition}.\\

\emph{Step 3: limsup inequality.} 
Let $u \in BV(\o; \{a,b\})$. We want to find a sequence $\{u_n\}_{n\in\N} \subset H^1(\o; \R^d)$ with $u_n\to u$ in $L^1(\o;\R^d)$ such that
\[
F^H_0(u) \geq \limsup_{n \to \infty} F_n(u_n).
\]
Since $F^H_0$ is the $\Gamma$-limit of $F^H_{\e}$ (again, because the constant $K_H$ is the same regardless of truncation) we can find a sequence $\{u_n\}_{n\in\N} \subset H^1(\o; \R^d)$ with $u_n\to u$ in $L^1(\o;\R^d)$ such that
\[
F^H_0(u) \geq \limsup_{n \to \infty} F^H_{\e_n}(u_n).
\]
Moreover, by our choice of truncation, $|u_n|\leq R$, so that $W(x,u_n(x))=\widetilde{W}(x,u_n(x))$ for a.e. $x\in\o$.
Note that
\[
\sup_{n\in\N}\int_{\o} \del_n |\nabla u_{n}|^2 \dx<+\infty
\]
and thus, we can apply Lemma \ref{lem:wHom} to conclude
\[
\limsup_{n \to \infty} F_{\e_{n}}(u_{n}) = \limsup_{n \to \infty} F^H_{\e_{n}}(u_{n})
\]
In particular, we have
\[
F^H_0(u) \geq \limsup_{n \to \infty} F_n(u_n).
\]

\end{proof}


\subsection*{Acknowledgement}

I would like to thank the Center for Nonlinear Analysis at Carnegie Mellon University for its support during the preparation of the manuscript. Adrian Hagerty was supported by the National Science Foundation under Grant No. DMS-1411646.


\bibliographystyle{siam}
\bibliography{Bibliography}

\begin{thebibliography}{1}

\bibitem{AFP}
{\sc L.~Ambrosio, N.~Fusco, and D.~Pallara}, {\em Functions of bounded
  variation and free discontinuity problems}, Oxford Mathematical Monographs,
  The Clarendon Press, Oxford University Press, New York, 2000.

\bibitem{AnsBraChi2}
{\sc N.~Ansini, A.~Braides, and V.~Chiad\`o~Piat}, {\em Gradient theory of
  phase transitions in composite media}, Proc. Roy. Soc. Edinburgh Sect. A, 133
  (2003), pp.~265--296.

\bibitem{Braides}
{\sc A.~Braides}, {\em {$\Gamma$}-convergence for beginners}, vol.~22 of Oxford
  Lecture Series in Mathematics and its Applications, Oxford University Press,
  Oxford, 2002.

\bibitem{CFHP18}
{\sc R.~Cristoferi, I.~Fonseca, A.~Hagerty, and C.~Popovici}, {\em A
  homogenization result in the gradient theory of phase transitions},
  arXiv:1808.01972,  (2018).

\bibitem{dalmaso93}
{\sc G.~Dal~Maso}, {\em An Introduction to {$\Gamma$}-Convergence}, Springer,
  1993.

\bibitem{fonseca89}
{\sc I.~Fonseca and L.~Tartar}, {\em The gradient theory of phase transitions
  for systems with two potential wells}, Proc. Roy. Soc. Edinburgh Sect. A, 111
  (1989), pp.~89--102.

\end{thebibliography}

\end{document}